\documentclass{article}%
\usepackage{amsmath}
\usepackage{amsfonts}
\usepackage{amssymb}
\usepackage{graphicx}
\usepackage[ruled,vlined,linesnumbered]{algorithm2e}%
\providecommand{\U}[1]{\protect\rule{.1in}{.1in}}
\newtheorem{theorem}{Theorem} [section]

\newtheorem{lemma}[theorem]{Lemma}

\newtheorem{problem}[theorem]{Problem}

\newenvironment{proof}[1][Proof]{\noindent\textbf{#1.} }{\ \rule{0.5em}{0.5em}}
\begin{document}

\author{Vadim E. Levit\\Department of Mathematics\\Ariel University, Ariel, Israel\\levitv@ariel.ac.il\\.\\David Tankus\\Department of Software Engineering\\Sami Shamoon College of Engineering, Ashdod, Israel\\davidt@sce.ac.il}
\title{Recognizing Relating Edges in Graphs without Cycles of Length $6$}
\date{}
\maketitle

\begin{abstract}
A graph $G$ is {\it well-covered} if all maximal independent sets are of the same cardinality. Let $w:V(G) \longrightarrow\mathbb{R}$ be a weight function. Then $G$ is {\it $w$-well-covered} if all maximal independent sets are of the same weight. An edge $xy \in E(G)$ is \textit{relating} if there exists an independent set $S$ such that both $S \cup \{x\}$ and $S \cup \{y\}$ are maximal
independent sets in the graph. If $xy$ is relating then $w(x)=w(y)$ for
every weight function $w$ such that $G$ is $w$-well-covered. Relating edges play an important role in investigating $w$-well-covered graphs.

The decision problem whether an edge in a graph is relating is
\textbf{NP-}complete \cite{bnz:related}. We prove that the problem remains \textbf{NP-}complete when the input is restricted to graphs without cycles of length $6$. This is an unexpected result because recognizing relating edges is known to be polynomially solvable for graphs without cycles of lengths $4$ and $6$ \cite{lt:relatedc4}, graphs without cycles of lengths $5$ and $6$
\cite{lt:wc456}, and graphs without cycles of lengths $6$ and $7$ \cite{tankus:c67}.

A graph $G$ belongs to the class $\mathbf{W_2}$ if every two pairwise disjoint independent sets in $G$ are included in two pairwise disjoint maximum independent sets \cite{Staples:thesis}. It is known that if $G$ belongs to the class $\mathbf{W_2}$ then it is well-covered. A vertex $v \in V(G)$ is {\it shedding} if for every independent set $S \subseteq V(G) \setminus N[v]$ there exists $u \in N(v)$ such that $S \cup \{u\}$ is independent \cite{w:shed}. Shedding vertices play an important role in studying the class $\mathbf{W_2}$. Recognizing shedding vertices is co-\textbf{NP-}complete, even when the input is restricted to triangle-free graphs \cite{lt:w2}. We prove that the problem is co-\textbf{NP-}complete for graphs without cycles of length $6$.
\end{abstract}

\section{Introduction}
\subsection{Definitions and Notation}
Throughout this paper $G$ is a simple (i.e., a finite, undirected,
loopless and without multiple edges) graph with vertex set $V(G)$ and edge set $E(G)$. Cycles of $k$ vertices are denoted by $C_{k}$. When we say that $G$ does not
contain $C_{k}$ for some $k \geq3$, we mean that $G$ does not admit subgraphs
isomorphic to $C_{k}$. It is important to mention that these subgraphs are not
necessarily induced. Let $\mathcal{G}(\widehat{C_{i_{1}}},..,\widehat{C_{i_{k}%
}})$ denote the family of all graphs which do not contain $C_{i_{1}}$,...,$C_{i_{k}}$.

Let $S\subseteq V(G)$ be a non-empty set of vertices, and let $i\in\mathbb{N}$. Then
\[
N_{i}(S)=\{v \in V(G)| \ min_{s\in S}\ d(v,s)=i\}, \ 
N_{i}[S]=\{v \in V(G)| \ min_{s\in S}\ d(v,s)\leq i\}
\]
where $d(x,y)$ is the minimal number of edges required to construct a path
between $x$ and $y$, or infinite if such a path does not exist. If $i\neq j$ then $N_{i}(S)\cap N_{j}(S) = \varnothing$. We abbreviate $N_{1}(S)$ and $N_{1}[S]$ to $N(S)$ and $N[S]$, respectively. If
$S=\{v\}$ for some $v\in V(G)$, then $N_{i}(\{v\})$, $N_{i}[\{v\}]$, $N(\{v\})$, and $N[\{v\}]$, are abbreviated to $N_{i}(v)$, $N_{i}[v]$, $N(v)$, and $N[v]$, respectively.

Let $T\subseteq V(G)$. Then $S$ \textit{dominates} $T$ if $T \subseteq N[S]$. If $N[S]=V(G)$ then $S$ dominates
the whole graph. The induced subgraph of $G$ with vertex set $S \subseteq V(G)$ is $G[S]$, and denote $G \setminus S = G[V(G) \setminus S]$.

\subsection{Relating Edges}
A set of vertices $S\subseteq V(G)$ is \textit{independent }if for every $x,y\in
S$, $x$ and $y$ are not adjacent. Obviously, an empty set is
independent. An independent set is called {\it maximal} if it is not contained in another independent set. An independent set is {\it maximum} if the graph does not contain an independent set with a higher cardinality.
A graph is called \textit{well-covered} if every maximal independent set is maximum. The problem of finding a maximum cardinality
independent set in an input graph is \textbf{NP}-hard. However, if the input is
restricted to well-covered graphs, then a maximum cardinality independent set
can be found polynomially using the \textit{greedy algorithm}.

Let $w:V(G) \longrightarrow\mathbb{R}$ be a weight function defined on the
vertices of $G$. For every set $S \subseteq V(G)$, define $w(S)=\Sigma_{s \in S}w(s)$. Then $G$ is {\it $w$-well-covered} if all maximal independent sets of $G$
are of the same weight. The set of weight functions $w$ for which $G$ is
$w$-well-covered is a \textit{vector space} \cite{cer:degree}. That vector space is denoted $WCW(G)$ \cite{bnz:related}.

The recognition of well-covered graphs is known to be \textbf{co-NP}-complete.
This was proved independently in \cite{cs:note} and \cite{sknryn:compwc}. The problem remains \textbf{co-NP}-complete even when the input is restricted to $K_{1,4}$-free graphs \cite{cst:structures}, or to circulant graphs \cite{bh:circulant}. However,
the problem is polynomially solvable for $K_{1,3}$-free graphs
\cite{tata:wck13f,tata:wck13fn}, for graphs with girth at least $5$
\cite{fhn:wcg5}, for graphs that contain neither $4$- nor $5$-cycles
\cite{fhn:wc45}, for graphs with a bounded maximal degree \cite{cer:degree},
and for chordal graphs \cite{ptv:chordal}.

Obviously, a graph $G$ is well-covered if and only if $w \equiv 1$ belongs to the vector space $WCW(G)$. Hence, for every family graphs, if recognizing well-covered graphs is \textbf{co-NP}-complete, then finding $WCW(G)$ is \textbf{co-NP}-hard. On the other hand, if for a specific family of graphs finding $WCW(G)$ can be completed polynomially then also recognizing well-covered graphs is a polynomial task. Polynomial algorithms which find  $WCW(G)$ are known for claw-free graphs \cite{lt:equimatchable} and for graphs without cycles of lengths 4, 5 and 6 \cite{lt:wc456}.

An edge $xy \in E(G)$ is \textit{relating} if there exists an independent set $S$ such that both $S \cup \{x\}$ and $S \cup \{y\}$ are maximal
independent sets in the graph. If $xy$ is relating then $w(x)=w(y)$ for
every weight function $w$ such that $G$ is $w$-well-covered. Relating edges play an important role in investigating $w$-well-covered graphs.

\begin{problem}
\label{reprob}
{\bf RE}\\
Input: A graph $G$, and an edge $e \in E(G)$.\\
Question: Is $e$ relating?
\end{problem}

A witness that $xy$ is a relating edge is an independent set $S$ such that both $S \cup \{x\}$ and $S \cup \{y\}$ are maximal
independent sets in the graph. 
The decision problem whether an edge in an input graph is relating is
\textbf{NP-}complete \cite{bnz:related}, and it remains \textbf{NP-}complete even
when the input is restricted to graphs without cycles of lengths $4$ and $5$
\cite{lt:relatedc4} or to bipartite graphs \cite{lt:npc}. However, recognizing relating edges can be done in polynomial time if the
input is restricted to graphs without cycles of lengths $4$ and $6$
\cite{lt:relatedc4}, to graphs without cycles of lengths $5$ and $6$
\cite{lt:wc456}, and to graphs without cycles of lengths $6$ and $7$ \cite{tankus:c67}.

\subsection{Shedding Vertices}
A vertex $v \in V(G)$ is {\it shedding} if for every independent set $S \subseteq V(G) \setminus N[v]$ there exists $u \in N(v)$ such that $S \cup \{u\}$ is independent. Equivalently, $v$ is shedding if there does not exist an independent set in $V(G) \setminus N[v]$ which dominates $N(v)$ \cite{w:shed}. It is easy to see that $v$ is shedding if and only if there does not exist an independent set in $N_{2}(v)$ which dominates $N(v)$. Shedding vertices are also called {\it extendable} \cite{fhn:wcg5}. 

\begin{problem}
\label{shedrob}
{\bf SHED}\\
Input: A graph $G$, and a vertex $v \in V(G)$.\\
Question: Is $v$ shedding?
\end{problem}

If $v$ is not shedding, a {\it witness} for being not shedding is a an independent set $S \subseteq N_{2}(v)$ which dominates $N(v)$. It is proved in \cite{lt:w2} that recognizing shedding vertices is co-\textbf{NP-}complete, even when the input is restricted to triangle-free graphs, but polynomial solvable for claw-free graphs, for graphs without cycles of length 5, and for graphs without cycles of lengths 4 and 6. Theorem \ref{shedrec6} shows the connection between the RE problem and the SHED problem.

\begin{theorem}
\cite{lt:w2}
\label{shedrec6}
Let $G$ be a graph without cycles of lengths 4, 5 and 6, and $xy \in E(G)$. Suppose $N(x) \cap N(y) = \varnothing$, $d(x) \geq 2$ and $d(y) \geq 2$. The following assertions are equivalent.
\begin{enumerate}
\item None of $x$ and $y$ is a shedding vertex.
\item $xy$ is a relating edge.
\end{enumerate}
\end{theorem}

A graph $G$ belongs to the class $\mathbf{W_{2}}$ if every two disjoint independent sets in $G$ are included in two disjoint maximum independent
sets \cite{Staples:thesis}. All graphs in the class $\mathbf{W_2}$ are well-covered. Recognizing $\mathbf{W_{2}}$ graphs is co-\textbf{NP-}complete \cite{fm:rem3w2}. Shedding vertices play an important role in studying the class $\mathbf{W_2}$ due to Theorem \ref{w2wcshed}.

\begin{theorem}
\label{w2wcshed}
\cite{lm:shed}
For every well-covered graph $G$ having no isolated vertices, the following assertions are equivalent:
\begin{enumerate}
\item $G$ is in the class $\mathbf{W_2}$.
\item All vertices of G are shedding.
\end{enumerate}
\end{theorem}

\subsection{SAT}
Let $\mathcal{X}=\{x_{1},...,x_{n}\}$ be a set of 0-1 variables. We define the
set of \textit{literals} $L_{\mathcal{X}}$ over $\mathcal{X}$ by
$L_{\mathcal{X}} = \{x_{i}, \overline{x_{i}} : i = 1,...,n\}$, where
$\overline{x} = 1 - x$ is the \textit{negation} of $x$. A \textit{truth
assignment} to $\mathcal{X}$ is a mapping $t:\mathcal{X} \longrightarrow \{0,1\}$ that assigns a value $t(x_{i}) \in\{0,1\}$ to each variable $x_{i}
\in\mathcal{X}$. We extend $t$ to $L_{\mathcal{X}}$ by putting $t(\overline{x_{i}}) = 1 - t(x_{i})$. A literal $l \in L_{\mathcal{X}}$ is true
under $t$ if $t(l) = 1$. A \textit{clause} over $\mathcal{X}$ is a conjunction
of some literals of $L_{\mathcal{X}}$, such that for every variable $x
\in\mathcal{X}$, the clause contains at most one literal out of $x$ and its
negation. Let $\mathcal{C}=\{c_{1},...,c_{m}\}$ be a set of clauses over $\mathcal{X}$. A truth assignment $t$ to $\mathcal{X}$ \textit{satisfies} a
clause $c_{j} \in\mathcal{C}$ if $c_{j}$ contains at least one true literal
under $t$. The number of times that a variable {\it appears} in $\mathcal{C}$ is the number of clauses that include the variable or its negation. 

\begin{problem}
\label{satprob}
{\bf SAT}\\
Input: A set of variables $\mathcal{X}=\{x_{1},...,x_{n}\}$, and a set of clauses $\mathcal{C}=\{c_{1},...,c_{m}\}$
over $\mathcal{X}$.\\
Question: Is there a truth assignment to $\mathcal{X}$ which satisfies all clauses of $\mathcal{C}$?
\end{problem}
The SAT problem is well-known to be \textbf{NP-}complete \cite{gj:NPC}. The number of times that a variable {\it appears} in an instance of SAT is the number of clauses which contain the variable or its negation. The SAT problem was learned thoroughly in recent years. The complexity statuses of many restricted cases of the problem were found. 

Horn SAT is a restricted case of the SAT problem where every clause contains at most one unnegeted literal. This problem is known to be polynomial solvable \cite{h:hornsat}. MONOTONE SAT is the SAT problem in the restricted case that every clause contains either negated literals or unnegated literals, but not both. MONOTONE SAT is \textbf{NP-}complete \cite{Gold1978}.

Let $k \geq 2$. The $k$-SAT problem is a restricted case of the SAT problem where each clause contains exactly $k$ different literals. For every $k \geq 3$, the $k$-SAT problem is well-known to be \textbf{NP-}complete \cite{gj:NPC}, while the 2-SAT problem is polynomial solvable \cite{eis:2sat}. In the MAX 2-SAT problem the input is a set of clauses of size 2 and a positive integer, $k$. It should be decided whether there exists a truth assignment which satisfies at least $k$ clauses. This problem was proved to be NP-complete \cite{gjs:maxsat2}. Moreover, even if every variable appears 3 times, the MAX 2-SAT problem is NP-complete \cite{h:maxsat}.

MONOTONE 3-SAT is NP-complete when each variable appears exactly 2 times negated and 2 times unnegated \cite{d:m223sat}. On the other hand, 3-SAT is always satisfiable, and therefore polynomial, when each variable appears at most 3 times \cite{tov:sat}. However, allowing clauses of size 2 and 3, with each variable appearing 3 times, is NP-complete \cite{papa:cc}.

The 1-in-3 SAT is another restricted case of the SAT problem, denoted X3SAT. In X3SAT every clause contains 3 literals. It should be decided whether there exists a truth assignment which satisfies exactly one literal in every clause. X3SAT is NP-complete even when all its variables occuring unnegated \cite{s:x3sat}.

\subsection{Main Results}
In Section 2 a new restricted version of the SAT problem, called 23SAT, is defined. It is proved that 23SAT is \textbf{NP}-complete. 

In Section 3 we prove that the SHED problem is co-\textbf{NP}-complete when its input is restricted to graphs without cycles of length 6. The proof is based on a polynomial reduction from the complement of the 23SAT problem. 

In Section 4 it is proved that the RE problem is \textbf{NP}-complete for graphs without cycles of length 6. The proof is based on a polynomial reduction from the complement of the SHED problem.

\section{23SAT}
Let $I = (\mathcal{X}, \mathcal{C})$ be an instance of SAT. A {\it major literal} is a literal in $L_{\mathcal{X}}$ that belongs to at least two clauses of $ \mathcal{C}$, while a {\it minor literal} belongs to only one clause. The following problem is a restricted case of the SAT problem.

\begin{problem}
\label{sat23prob}
{\bf 23SAT}\\
Input: A set of variables $\mathcal{X}=\{x_{1},...,x_{n}\}$, and a set of clauses $\mathcal{C}=\{c_{1},...,c_{m}\}$ over $\mathcal{X}$ such that every clause contains 2 or 3 literals, and every clause contains at most 1 major literal.\\
Question: Is there a truth assignment to $\mathcal{X}$ which satisfies all clauses of $\mathcal{C}$?
\end{problem}

\begin{theorem}
\label{sat23npc} 23SAT is \textbf{NP}-complete.
\end{theorem}

\begin{proof} 
Clearly, the problem is \textbf{NP}. We prove \textbf{NP}-completeness by a polynomial reduction from SAT. Let $I_{1} = (\mathcal{X}=\{x_{1},...,x_{n}\}, \mathcal{C}=\{c_{1},...,c_{m}\})$ be an instance of SAT. For every $1 \leq j \leq m$, denote $c_{j} = \{l_{j,1},\ldots,l_{j,k(j)}\}$, $k(j) \geq 2$, and define 
\[f(c_{j})=\{\{l_{j,1},y_{j,1}\}, \{\overline{y_{j,1}},l_{j,2},y_{j,2}\},\{\overline{y_{j,2}},l_{j,3},y_{j,3}\},\ldots,\{\overline{y_{j,k(j)-1}},l_{j,k(j)}\}\}\]
where $y_{j,1},\ldots,y_{j,k(j)-1}$ are new variables. Clearly, $y_{j,1},\overline{y_{j,1}},\ldots,y_{j,k(j)-1},\overline{y_{j,k(j)-1}}$ are minor literals. Let 
\[\mathcal{X}' = \mathcal{X} \cup \{y_{j,r} : 1 \leq j \leq m, 1 \leq r \leq k(j)-1\}, \  \ \mathcal{C}' = \bigcup_{1 \leq j \leq m} f(c_{j}) \]
Then $I_{2} = (\mathcal{X}', \mathcal{C}')$ is an instance of 23SAT, since the size of every clause is 2 or 3, and every clause contains at most one major literal. It remains to prove that $I_{1}$ and $I_{2}$ are qeuivalent. 

If $I_{1}$ is positive then there exists a truth assignment $t_{1} : \mathcal{X} \longrightarrow \{0,1\}$ which satisfies $\mathcal{C}$. For every $1 \leq j \leq m$ the fact that $t_{1}$ satisfies $c_{j}$ implies that there exists $1 \leq r(j) \leq k(j)-1$ such that $t_{2}(l_{j,r(j)}) = 1$. Let $t_{2} : \mathcal{X}' \longrightarrow \{0,1\}$ be the following extraction of $t_{1}$. Define $t_{2}(y_{j,r'}) = 1$ for every $1 \leq r' < r(j)$ and $t_{2}(y_{j,r'}) = 0$ for every $r(j) \leq r' \leq k(j)-1$. Clearly, $t_{2}$ satisfies all clauses of $f(c_{j})$ for every $1 \leq j \leq m$. Consequently, $I_{2}$ is positive.

Conversely, if $I_{2}$ is positive then there exists a truth assignment $t_{2} : \mathcal{X}' \longrightarrow \{0,1\}$ which satisfies $\mathcal{C}'$. Hence, $t_{2}$ satisfies $f(c_{j})$ for every $1 \leq j \leq m$. However, $f(c_{j})$ contains $k(j)$ clauses, and $k(j)-1$ variables form $\mathcal{X}' \setminus \mathcal{X}$. Therefore, there exists $1 \leq r(j) \leq k(j)-1$ such that $t_{2}(l_{j,r(j)}) = 1$. Define $t_{1} : \mathcal{X} \longrightarrow \{0,1\}$ by $t_{1}(x_{i})=t_{2}(x_{i})$ for every $1 \leq i \leq n$. Then $t_{1}$ satisfies $c_{j}$ because $t_{1}(l_{j,r(j)}) = 1$ for every $1 \leq j \leq m$. Consequently, $t_{1}$ satisfies $\mathcal{C}$, and $I_{1}$ is positive.
\end{proof} 

Let $I=(\mathcal{X}, \mathcal{C})$ be an instance of 23SAT. A {\it bad pair} in $I$ is a set of 2 clauses $\{c_{1}, c_{2}\} \subseteq \mathcal{C}$ such that there exist 2 literals, $l_{1}$ and $l_{2}$, for which $\{l_{1},l_{2}\} \subseteq c_{1}$ and $\{\overline{l_{1}}, \overline{l_{2}}\} \subseteq c_{2}$.  

\begin{theorem}
\label{sat23withoutc6} There exists a polynomial algorithm which receives as its input an instance of 23SAT, and finds an equivalent instance of 23SAT without bad pairs.
\end{theorem}

\begin{proof} 
The following algorithm receives as its input an instance $I_{1}$ of 23SAT with bad pairs. The algorithm finds an equivalent instance, $I_{2}$, of 23SAT such that the number of bad pairs in $I_{2}$ is smaller than the number of bad pairs in $I_{1}$. By invoking the algorithm repeatedly, one can find an instance of 23SAT which is equivalent to the original one, and does not contain bad pairs.

Denote $I_{1}=(\mathcal{X}, \mathcal{C})$. There exist clauses $c_{1}, c_{2} \in \mathcal{C}$, and literals, $l_{1}$ and $l_{2}$, such that $\{l_{1},l_{2}\} \subseteq c_{1}$ and $\{\overline{l_{1}}, \overline{l_{2}}\} \subseteq c_{2}$. By definition of 23SAT, every clause contains at most one major literal. Assume without loss of generality that $l_{1}$ is a minor literal. At least one of $\overline{l_{2}}$ and $\overline{l_{1}}$ is a minor literal. 

If $\overline{l_{2}}$ is a minor literal then construct a truth assignment $t$ for $I_{1}$ by assigning $t(l_{1}) = t(\overline{l_{2}}) = 0$. This assigment satisfies both $c_{1}$ and $c_{2}$. Moreover, if $I_{1}$ contains other clauses with $\overline{l_{1}}$ and $l_{2}$, these clauses are satisfied, too. Let $I_{2}$ be the instance of 23SAT obtained from $I_{1}$ by omitting all clauses which contain $\overline{l_{1}}$ or $l_{2}$. 

We show that $I_{1}$ and $I_{2}$ are equivalent. If there exists a truth assignment $t_{2}$ that satisfies $I_{2}$, then assign $t(x) = t_{2}(x)$ for every variable $x$ of $I_{2}$. Clearly, $t$ satisfies $I_{1}$. On the other hand, since every clause of $I_{2}$ belongs also to $I_{1}$, if there does not exist a truth assigment that satisfies $I_{2}$, then there does not exist a truth assignment which satisfies $I_{1}$. Therefore, $I_{1}$ and $I_{2}$ are equivalent, and the number of bad pairs in $I_{2}$ is smaller than the number of bad pairs in $I_{1}$.

Suppose $\overline{l_{1}}$ is a minor literal. Assigning $l_{1}=\overline{l_{2}}$ satisfies both $c_{1}$ and $c_{2}$, and does not affect the other clauses of $I_{1}$. Hence, let $I_{2}$ be the instance of 23SAT obtained from $I_{1}$ by omitting $c_{1} $ and $c_{2}$. Clearly, $I_{2}$ is equivalent to $I_{1}$, and the number of bad pairs in $I_{2}$ is smaller than the number of bad pairs in $I_{1}$.
\end{proof}

\section{Shedding Vertices}
Let $I = (\mathcal{X}=\{x_{1},...,x_{n}\}, \mathcal{C}=\{c_{1},...,c_{m}\})$ be an instance of SAT. Define $G_{I}$ to be the following graph. 
\[V(G_{I}) = \{v\} \cup \{ w_{j} : 1 \leq j \leq m\} \cup \{u_{i}, u_{i}' : 1 \leq i \leq n\}\] 
\[E(G_{I})= \{vw_{j} : 1 \leq j \leq m\} \cup \{w_{j}u_{i} : x_{i} \in c_{j}\} \cup \{w_{j}u_{i}' : \overline{x_{i}} \in c_{j}\} \cup \{u_{i}u_{i}' : 1 \leq i \leq n\}\]
Note that the subgraph induced by $N_{2}(v)$ is a {\it matching}, i.e. a disjoint union of copies of $K_{2}$. Every maximal independent set of $N_{2}(v)$ contains exactly one of $u_{i}$ and $u_{i}'$, for every $1 \leq i \leq n$.

\begin{lemma}
\label{satreduction} An instance $I$ of SAT is satisfiable if and only if there exists in $G_{I}$ an independent set $S \subseteq N_{2}(v)$ which dominates $N(v)$.
\end{lemma}

\begin{proof} 
Suppose that there exists an independent set $S$ of $N_{2}(v)$ which dominates $N(v)$. Define a truth assignment $t : \mathcal{X} \longrightarrow \{0,1\}$ as follows. For every $1 \leq i \leq n$, $t(x_{i}) = 1$ if and only if $u_{i} \in S$. Otherwise, $t(x_{i})=0$. The fact that $S$ dominates all vertices of $N(v)$ implies that all clauses of $\mathcal{C}$ are satisfied by $t$.

On the other hand, assume that there exists a truth assignment $t : \mathcal{X} \longrightarrow \{0,1\}$ which satisfies $\mathcal{C}$. Define $S = \{u_{i} : t(x_{i})=1\} \cup \{u_{i}' : t(x_{i})=0\}$. Clearly, $S \subseteq N_{2}(v)$. The fact that for every $1 \leq i \leq n$ the set $S$ contains exactly one of  $u_{i}$ and $u_{i}'$ implies that $S$ is independent. The fact that $t$ satisfies all clauses of $\mathcal{C}$ implies that $S$ dominates all vertices of $N(v)$.
\end{proof}

\begin{lemma}
\label{c6free} Let $I = (\mathcal{X}, \mathcal{C})$ be an instance of 23SAT without bad pairs. Then $G_{I} \in \mathcal{G}(\widehat{C}_{6})$.
\end{lemma}

\begin{proof}
Assume on the contrary that $v$ belongs to a cycle of length 6. Then there exist vertices $w_{1}$, $w_{2}$, $w_{3}$ in $N(v)$ and $z_{1}$, $z_{2}$ in $N_{2}(v)$ such that $(v,w_{1},z_{1},w_{2},z_{2},w_{3})$ is a cycle of length 6. Let $l_{1}$ and $l_{2}$ be the literals which represent $z_{1}$ and $z_{2}$, respectively. Let $c_{1}$, $c_{2}$ and $c_{3}$ be the clauses which represent $w_{1}$, $w_{2}$ and $w_{3}$, respectively. Then $l_{1}$ is a major literal since it belongs to both $c_{1}$ and $c_{2}$. Similarly, $l_{2}$ is a major literal, since it belongs to both $c_{2}$ and $c_{3}$. Hence $c_{2}$ contains two major literals, which is a contradiction. Therefore, $v$ is not a part of a cycle of length 6.

Assume on the contrary that $G_{I} \setminus \{v\}$ contains a cycle $C$ of length 6. Since $N(v)$ is independent, $|C \cap N(v)| \leq 3$. The fact that every connected component of $N_{2}(v)$ is $K_{2}$ implies that $|C \cap N(v)| \geq 2$. 

If $|C \cap N(v)| = 3$ and $|C \cap N_{2}(v)| = 3$ then the vertices in $C \cap N_{2}(v)$ represent major literals, and the vertices of $C \cap N(v)$ represent clauses that each of them contains at least 2 major literals, which is a contradiction. Consequently, $|C \cap N(v)| = 2$ and $|C \cap N_{2}(v)| = 4$. That means that either $C = (u_{1},u_{1}',w_{1},u_{2},u_{2}',w_{2})$ or $C = (u_{1},u_{1}',w_{1},u_{2}',u_{2},w_{2})$ for $u_{1}$, $u_{1}'$, $u_{2}$, $u_{2}'$ in $N_{2}(v)$ and $w_{1}$, $w_{2}$ in $N(v)$. Therefore, there exist clauses $c_{1}$ and $c_{2}$ and literals $l_{1}$ and $l_{2}$ such that $\{l_{1},l_{2}\} \subseteq c_{1}$ and $\{\overline{l_{1}},\overline{l_{2}}\} \subseteq c_{2}$. Hence, $I$ contains a bad pair, which is a contradiction. $G_{I} \in \mathcal{G}(\widehat{C}_{6})$.
\end{proof}

\begin{theorem}
\label{shedc6conpc} Recognizing shedding vertices is \textbf{co-NP}-complete, even for graphs without cycles of length 6.
\end{theorem}

\begin{proof}
It should be proved that the complement problem is \textbf{NP}-complete. An instance of the problem is $I = (G, v)$, where $G$ is a graph without cycles of length 6, and $v \in V(G)$. The instance is positive if and only if there exists an independent set in $N_{2}(v)$ which dominates $N(v)$.

The problem is obviously \textbf{NP}. We prove \textbf{NP}-completeness by a polynomial reduction from 23SAT. Let $I_{1}$ be an instance of 23SAT. By Theorem \ref{sat23withoutc6}, it is possible to find in polynomial time an equivalent instance $I_{2}$ of 23SAT without bad pairs. Let $I_{3} = (G_{I_{2}},v)$ be an instance of the complement of the SHED problem. By Lemma \ref{c6free}, $G_{I_{2}} \in \mathcal{G}(\widehat{C}_{6})$. By Lemma \ref{satreduction}, $I_{2}$ is satisfiable if and only if there exists an independent set in $N_{2}(v)$ which dominates $N(v)$.
\end{proof}

\section{Relating Edges}
Theorem \ref{rec6} is the main result of this section.

\begin{theorem}
\label{rec6} 
Recognizing relating edges is \textbf{NP}-complete for graphs without cycles of length 6.
\end{theorem}

\begin{proof}
Clearly, the problem is \textbf{NP}. We prove \textbf{NP}-completeness by a polynomial reduction from the complement of the SHED problem for graphs without cycles of length 6. Let $I_{1} = (G,x)$ be an instance of the comlement of SHED such that $G \in \mathcal{G}(\widehat{C}_{6})$. Then $I_{1}$ is positive if and only if there exists an independent set in $N_{2}(x)$ which dominates $N(x)$. Define a new graph $G'$ as follows. $V(G') = V(G) \cup \{y\}$, when $y$ is a new vertex, and $E(G') = E(G) \cup \{xy\}$. The fact that $G \in \mathcal{G}(\widehat{C}_{6})$ implies $G' \in \mathcal{G}(\widehat{C}_{6})$. Let $I_{2} = (G', xy)$ be an instance of the RE problem. It remains to prove that $I_{1}$ and $I_{2}$ are equivalent.

Suppose $I_{1}$ is positive. The graph $G$ contains an independent set $S \subseteq N_{2}(x)$ which dominates $N(x)$. Let $S^{*}$ be a maximal independent set of $G \setminus \{x\}$ which contains $S$. In the graph $G'$ the sets $S^{*} \cup \{x\}$ and $S^{*} \cup \{y\}$ are maximal independent sets. Therefore, $S^{*}$ is a witness that $xy$ is relating, and $I_{2}$ is positive.

Conversely, if $I_{2}$ is positive then there exists an independent set $S \subseteq V(G')$ such that $S \cup \{x\}$ and $S \cup \{y\}$ are maximal independent sets of $G'$. Obviously, $S \cap N_{2}(x)$ is an independent set of $N_{2}(x)$ which dominates $N(x)$, and $I_{1}$ is positive.
\end{proof}

\section{Conclusions}
We proved that the RE problem is \textbf{NP}-complete for graphs without cycles of length 6. This result is suprising and unexpected since the RE problem is polynomially solvable for graphs without cycles of lengths 6 and 4 \cite{lt:relatedc4}, for graphs without cycles of lengths 6 and 5 \cite{lt:wc456}, and for graphs without cycles of lengths 6 and 7 \cite{tankus:c67}. Each of the algorithms presented in the three above-mentioned papers works as follows. It finds an independent set $S_{x}$ in $N_{2}(x) \cap N_{3}(y)$ which dominates $N(x) \cap N_{2}(y)$. Then it finds similarly an independent set $S_{y}$ in $N_{2}(y) \cap N_{3}(x)$ which dominates $N(y) \cap N_{2}(x)$. Since the graph does not contain cycles of length 6, there are no edges between $S_{x}$ and $S_{y}$. Hence, if both $S_{x}$ and $S_{y}$ exist then $S = S_{x} \cup S_{y}$ is an independent set which is a witness that $xy$ is relating, and the algorithm terminates announcing that the edge is relating. On the other hand, if at least one of $S_{x}$ and $S_{y}$ does not exist then the instance of the problem is negative. First, we conjectured that an algorithm for graphs without cycles of length 6 would work according to the same principle, and thus generalize the last three results. However, we found out that the existence of such an algorithm would imply that {\bf P=NP}.

Theorem \ref{shedrec6} shows a connection between the RE and SHED problems. In each of these problems it should be decided whether there exists an independent set in a subset of $V(G)$ which dominates another subset of $V(G)$. Hence, for many families of graphs either RE is \textbf{NP}-complete and SHED is co-\textbf{NP}-complete, or both problems are polynomially solvable. That holds also for graphs without cycles of length 6, the family of graphs which is studied in this paper. However, some families are exceptional. For example, concerning graphs without cycles of lengths 4 and 5, RE is \textbf{NP}-complete
\cite{lt:relatedc4}, but SHED is a polynomially solvable \cite{lt:w2}.

\end{document}